\documentclass[letterpaper, 10 pt, conference]{ieeeconf} %
\pdfoutput=1
\IEEEoverridecommandlockouts
\usepackage{cite}

\usepackage{amsthm,amsmath,amssymb,amsfonts}
\usepackage{bm}  
\usepackage{algorithmic}
\usepackage{algorithm}

\usepackage{graphicx}
\graphicspath{{figures/}}
\usepackage{xcolor}

\usepackage{array}

\makeatletter
\let\MYcaption\@makecaption
\makeatother

\usepackage[font=footnotesize]{subcaption}

\makeatletter
\let\@makecaption\MYcaption
\makeatother

\usepackage{textcomp}
\def\BibTeX{{\rm B\kern-.05em{\sc i\kern-.025em b}\kern-.08em
    T\kern-.1667em\lower.7ex\hbox{E}\kern-.125emX}}

\makeatletter
\def\input@path{{tex_files/}}
\makeatother

\theoremstyle{plain}
\newtheorem{theorem}{Theorem}
\newtheorem{lemma}{Lemma}
\newtheorem{corollary}{Corollary}

\theoremstyle{definition}

\def\({\left(}
\def\){\right)}
\def\[{\left[}
\def\]{\right]}

  \def\dbf{{\bf d}}

\def\Kbf{{\bf K}}

\def\Rbf{{\bf R}}

  \def\ubf{{\bf u}}

  \def\xbf{{\bf x}}
  \def\ybf{{\bf y}}

\def\Xibf{{\bm{\Xi}}}

\def\Phibf{\bm{\Phi}}

\def\0bf{{\bf 0}}
\def\1bf{{\bf 1}}

\def\Rmbb{\mathbb{R}}

\def\Acal{\mathcal{A}}

\def\Hcal{\mathcal{H}}

\def\Rcal{\mathcal{R}}  
\def\Scal{\mathcal{S}}

\newif\ifshowWriterComment
\newcommand\writercomment[3]{\expandafter\newcommand\csname #2\endcsname[1]{\ifshowWriterComment{\color{#3} (#1: ##1)}\fi}}

\writercomment{Shih-Hao}{SHT}{blue}
\writercomment{Lauren}{LC}{green}

\def\Pbfxx{\Phibf_{\xbf\xbf}}
\def\Pbfux{\Phibf_{\ubf\xbf}}
\def\Pbfxy{\Phibf_{\xbf\ybf}}
\def\Pbfuy{\Phibf_{\ubf\ybf}}

\def\Pxx{\Phi_{xx}}
\def\Pux{\Phi_{ux}}
\def\Pxy{\Phi_{xy}}
\def\Puy{\Phi_{uy}}

\newcommand{\OFqple}{\{\Pbfxx, \Pbfux, \Pbfxy, \Pbfuy\}}

\newcommand{\norm}[1]{\left\lVert #1 \right\rVert}
\def\nul#1{\text{\rm null}\left(#1\right)}
\def\basis#1{\text{\rm basis}\left(#1\right)}
\def\normalize#1{\text{\rm normalize}\left(#1\right)}

\def\tdA{\tilde{A}}
\def\tdB{\tilde{B}}

\renewcommand{\paragraph}[1]{\vspace*{0.3\baselineskip}\noindent {\bf #1:} }

\DeclareMathOperator*{\argmin}{argmin}

\renewcommand{\vec}[1]{\overrightarrow{#1}}
\newcommand{\unvec}[1]{\overleftarrow{#1}}

\DeclareMathOperator{\xx}{\phi_{xx}}
\DeclareMathOperator{\xy}{\phi_{xy}}
\DeclareMathOperator{\ux}{\phi_{ux}}
\DeclareMathOperator{\uy}{\phi_{uy}}
\DeclareMathOperator{\nullspace}{null}

\def\mat#1{\begin{bmatrix}#1\end{bmatrix}}

\def\alg#1{Algorithm~\ref{alg:#1}}

\def\fig#1{Fig.~\ref{fig:#1}}

\def\lem#1{Lemma~\ref{lem:#1}}

\def\sec#1{Section~\ref{sec:#1}}

\def\thm#1{Theorem~\ref{thm:#1}}
\def\eqn#1{\eqref{eqn:#1}}

\def\st{{\rm s.t.}}
\def\OptConsSep{&&\quad}

\makeatletter
\newcommand{\OptMin}{\@ifstar\OptMinStar\OptMinNoStar}
\makeatother
\newcommand{\OptMinStar}[3][]{%
\ifx\\#1\\ \else\begin{subequations}\label{eqn:#1}\fi%
\begin{alignat}{2}%
\min\ &\ #2 \nonumber \\%
\st\ #3 %
\end{alignat}%
\ifx\\#1\\ \else\end{subequations}\fi%
}
\newcommand{\OptMinNoStar}[3][]{%
\ifx\\#1\\ \else\begin{subequations}\label{eqn:#1}\fi%
\begin{alignat}{2}%
\min\ &\ #2 \ifx\\#1\\ \nonumber \else \tag{\ref{eqn:#1}}\label{eqnset:#1} \fi \\%
\st\ #3 %
\end{alignat}%
\ifx\\#1\\ \else\end{subequations}\fi%
}

\newcommand\OptCons[3]{%
&\ #1 %
\ifx\\#2\\ \else \OptConsSep #2 \fi%
\ifx\\#3\\ \nonumber \else \label{eqn:#3} \fi%
}

\newcommand\OptConsN[2]{%
&\ #1 %
\ifx\\#2\\ \else \OptConsSep #2 \fi%
}

\title{\LARGE \bf Output-Feedback System Level Synthesis via Dynamic Programming}

\pdfoutput=1
\author{Lauren E. Conger and Shih-Hao Tseng
\thanks{Lauren E. Conger and Shih-Hao Tseng are with the Division of Engineering and Applied Science, California Institute of Technology, Pasadena, CA 91125, USA.  Emails: {\tt\small \{lconger,shtseng\}@caltech.edu}}
}

\showWriterCommenttrue

\begin{document}

\maketitle
\thispagestyle{empty}
\pagestyle{empty}

\bstctlcite{IEEE_BSTcontrol}

\begin{abstract}
System Level Synthesis (SLS) allows us to construct internally stabilizing controllers for large-scale systems. However, solving large-scale SLS problems is computationally expensive and the state-of-the-art methods consider only state feedback; 
output feedback poses additional challenges because the constraints are no longer uniquely row or column separable.

We exploit the structure of the output-feedback SLS problem by vectorizing the multi-sided matrix multiplications in the SLS optimization constraints, which allows us to reformulate it as a discrete-time control problem and solve using two stages of dynamic programming (DP). Additionally, we derive an approximation algorithm that offers a faster runtime 
by partially enforcing the constraints, and show that this algorithm offers the same results. DP solves SLS up to $7$ times faster, with an additional $42\%$ to $68\%$ 
improvement using the approximation algorithm, than a convex program solver, and scales with large state dimensions 
and finite impulse response horizon.



\end{abstract}

\section{Introduction}\label{sec:introduction}

In classical control theory, transfer functions and controllers are constructed to minimize a variety of cost functions, primarily functions of the state and input \cite{doyle1992feedback} \cite{zhou1996robust}. These 
celebrated techniques have been proven useful for state-feedback small-scale systems. However, in modern applications of linear control, especially for large-scale systems, the full state of a system is often unknown or too expensive to acquire within a reasonable amount of time. 
As a result, we rely on partial observations of the state to control the system in an output-feedback fashion. 
Not only does such a proxied information structure make output-feedback control difficult, but it also hinders one in incorporating additional system level constraints.
To address these challenges, system level synthesis (SLS) \cite{anderson2019system, wang2019system} 
derives its system level parameterization of output-feedback controllers as a convex optimization problem that
 maintains system-level constraint enforcement. This formulation can be used to solve the linear-quadratic Gaussian problem and optimization under the $\Hcal_2$ norm as shown in \cite{anderson2019system}. With SLS, we can 
 solve these traditional control problems with added system-level constraints, which is not possible with the Youla or Input-Output parameterizations \cite{furieri2019input} \cite{sabau2014youla} \cite{youla1976modern2}.

Although convex optimization problems are in general tractable, it is still slow to solve the SLS problems for output-feedback settings.
Existing work proposes to accelerate the SLS computation by exploiting the structure of underlying systems, and an early work \cite{wang2018separable} has established that an SLS problem can be solved by ADMM in theory if the problem is partially separable -- for both the constraints and the objective. In practice, most work primarily focuses on state-feedback settings \cite{anderson2017structured,matni2017scalable,amo-alonso2020distributed,amo-alonsosubeffective} since the state-feedback SLS constraints are column-wise separable and we only need to ensure the objective is also column-wise separable. Output-feedback SLS is much more complicated as its constraints consist of both column-wise and row-wise separable parts, which require the objective to be both column-wise and row-wise separable for ADMM to work. As a consequence, we end up relying on the solver to deal with general non-separable output-feedback problems.

Some recent work provides an alternative approach to the SLS computation. \cite{tseng2020system} shows that a state-feedback SLS problem can be reformulated as a discrete-time control problem, which admits efficient computation via dynamic programming (DP). This approach could potentially be applied to output-feedback SLS as long as we could address the following discrepancy between the state-feedback and output-feedback SLS. First, the state-feedback SLS system response comprises only two transfer matrices, which are then conveniently set as the state and the control in \cite{tseng2020system}, while the output-feedback version has four matrices to handle. Also, the system dynamics matrices are multiplied at the left in state-feedback constraints, while the output-feedback formulation have both left and right matrix multiplications in its constraints. Further, we have only one initial condition in the state-feedback SLS while the initial condition has a much more complicated structure in the output-feedback setting. 

\subsection{Contributions and Organization}
We address the aforementioned challenges and extend the approach in \cite{tseng2020system} to derive a DP algorithm for output-feedback SLS by reformulating it as a discrete-time control problem. By our choice of the corresponding state and control variables, the control problem consists of a new transition constraint and initial condition compared to the state-feedback scenario. By vectorizing the multi-sided matrix multiplications, we can solve the control problem by a two-stage DP procedure. We develop the full DP and an approximation algorithm by skipping the constraint enforcement up to some \emph{allowance}. Explicitly, we derive the algorithm for $\Hcal_2$ and quadratic objectives. By comparing the computing time with the CVX-based solver \cite{SLSpy}, we show that DP perform up to $7$ times faster. Also, the approximation algorithm can further shorten the computing time by $42\%$ to $68\%$ with mild accuracy degradation.

The paper is organized as follows. We begin with brief introductions of output-feedback SLS and DP in \sec{preliminaries}. 
Then, we present our DP algorithms to solve and approximate output-feedback SLS in \sec{DP_algorithm}. In \sec{examples}, 
the $\Hcal_2$ and quadratic costs are explicity derived, and we numerically evaluate our algorithms in \sec{evaluation}. 
We conclude the paper in \sec{conclusion}.

\subsection{Notation} 

\def\Rp{\Rcal_{p}}
\def\Rsp{\Rcal_{sp}}
\def\RHinf{\Rcal\Hcal_{\infty}}

Let $z^{-1}\RHinf$ 
and $\RHinf$ denote the set of strictly proper 
and stable proper transfer matrices, respectively, all defined according to the underlying setting, continuous or discrete.
Lower- and upper-case letters (such as $x$ and $A$) denote vectors and matrices respectively, while bold lower- and upper-case characters and symbols (such as $\ubf$ and $\Rbf$) are reserved for signals and transfer matrices. $I$ and $0$ are the identity matrix and all-zero matrix/vector (with dimensions defined according to the context).
We denote by $A^+$ the pseudo inverse (Moore-Penrose inverse) of $A$ and $\normalize{A}$ the matrix containing all normalized non-zero rows in $A$. Let $\vec{A}$ be the vectorized matrix $A$, which stacks the columns of $A$, and let $\unvec{x}$ be the inverse operation such that $x = \vec{A}$ is equivalent to $A = \unvec{x}$. The null space of a matrix $\Psi$ is written as $\nul{\Psi} = \{v : \Psi v = 0\}$, where $0$ is an all-zero vector. We denote by $\basis{\Scal}$ a matrix where its columns form a basis that spans the linear (sub)space $\Scal$.

\section{Preliminaries}\label{sec:preliminaries}  

\subsection{Output-Feedback System Level Synthesis (SLS)}

System level synthesis (SLS) allows us to design a system's closed-loop response by solving a convex optimization problem \cite{wang2019system,anderson2019system}. For an output-feedback system evolving according to the following dynamics
\begin{align*}
x[t+1] =&\ A x[t] + B u[t] + d_x[t] \\
y[t] =&\ C x[t] + D u[t] + d_y[t]
\end{align*}
where $x[t]$ is the state, $u[t]$ the control, $y[t]$ the output (measurement), and $d_x[t]/d_y[t]$ the disturbances at time $t$, the closed-loop system is parameterized by
\begin{align*}
    \begin{bmatrix}
        \xbf \\ \ubf
    \end{bmatrix} = 
    \begin{bmatrix}
        \Pbfxx & \Pbfxy \\ \Pbfux & \Pbfuy
    \end{bmatrix} 
    \begin{bmatrix}
        \dbf_x \\ \dbf_y
    \end{bmatrix}.
\end{align*}

SLS synthesizes an internally stabilizing controller $\Kbf$ in the frequency domain by solving the convex optimization
\OptMin[SLS_output_feedback]{ 
g(\Pbfxx,\Pbfxy,\Pbfux,\Pbfuy)
}{
\OptCons{
\mat{zI-A & -B}
\mat{
\Pbfxx & \Pbfxy \\
\Pbfux & \Pbfuy
}
=
\mat{I & 0},
}{}{}\\
\OptCons{
\mat{
\Pbfxx & \Pbfxy \\
\Pbfux & \Pbfuy
}
\mat{zI-A \\ -C}
=
\mat{I \\ 0},
}{}{}\\
\OptCons{
\Pbfxx,\Pbfxy,\Pbfux \in z^{-1}\RHinf, \Pbfuy \in \RHinf
}{}{stability-constraint}\\
\OptCons{
\mat{
\Pbfxx & \Pbfxy \\
\Pbfux & \Pbfuy
} \in \Scal,
}{}{}
}
where $g$ is the objective and $\Scal$ the set of system level constraints.
We call $\OFqple$ the system response, and a solution to \eqn{SLS_output_feedback} leads to the desired controller \cite[Corollary 5]{tseng2021realization}
\begin{align*}
\Kbf = \Kbf_0 \left( I + D\Kbf_0 \right)^{-1}
\end{align*}
where $\Kbf_0 = \Pbfuy - \Pbfux \Pbfxx^{-1} \Pbfxy$.
The control then follows from $\ubf = \Kbf \ybf$ where $\ubf$ and $\ybf$ are the frequency domain signals of the control $u$ and output $y$.

In this paper, we consider finite impulse response (FIR) solutions with horizon $T$, i.e.,
\begin{align*}
\Pbfxx = \sum\limits_{\tau = 1}^T z^{-\tau} \Pxx[\tau], \quad
\Pbfxy = \sum\limits_{\tau = 1}^T z^{-\tau} \Pxy[\tau], \\
\Pbfux = \sum\limits_{\tau = 1}^T z^{-\tau} \Pux[\tau], \quad
\Pbfuy = \sum\limits_{\tau = 0}^T z^{-\tau} \Puy[\tau],
\end{align*}
where $\Phi_{\bullet}[\tau]$ are the corresponding spectral components.
Notice that the summation limit is different for $\Pbfuy$ due to \eqn{stability-constraint}. Further, we assume that the objective $g$ is a finite sum of per-step costs:
\begin{align*}
&\ g(\Pbfxx,\Pbfxy,\Pbfux,\Pbfuy)\\
=&\ \sum\limits_{\tau = 0}^{T} g_{\tau}(\Pxx[\tau],\Pxy[\tau],\Pux[\tau],\Puy[\tau]).
\end{align*}

\subsection{Dynamic Programming (DP)}\label{sec:preliminaries-DP}

Dynamic programming (DP) breaks down a cost-optimization problem into a series of subproblems recursively correlated to each other. The optimal solution to the original problem can then be obtained by iteratively solving the subproblems. Such a breakdown is possible especially when the overall cost $h$ is the sum of per-step costs $h_\tau$, e.g.,
\begin{align*}
h(\xbf,\ubf) = \sum\limits_{\tau=0}^T h_\tau(x[\tau],u[\tau]).
\end{align*}
Letting $x[\tau + 1] = f(x[\tau], u[\tau])$, DP then derives the cost-to-go functions $V_\tau(x[\tau])$ from the Bellman equation:
\begin{align*}
V_\tau(x[\tau]) = \min\limits_{\hat{u} \in \Acal_u[\tau]} h_\tau(x[\tau],\hat{u}) + V_{\tau + 1}(f(x[\tau],\hat{u}))
\end{align*}
where $\Acal_u[\tau]$ is the set of admissible $u[\tau]$ and $V_{T+1}(x[\tau]) = 0$. As such, given the initial condition $x[0]$, the optimal cost of the original problem is $V_0(x[0])$, and the optimal $u[\tau]$ is the one achieving the minimum of the cost-to-go function $V_\tau(x[\tau])$, i.e.,
\begin{align}
u[\tau] =&\ \argmin\limits_{\hat{u} \in \Acal_u[\tau]} h_\tau(x[\tau],\hat{u}) + V_{\tau + 1}(f(x[\tau],\hat{u})) \nonumber \\
=&\ K_\tau(x[\tau]).
\label{eqn:DP-U}
\end{align}
Meanwhile, we can also use the above $u[\tau]$ to express $V_\tau(x[\tau])$ as
\begin{align}
V_\tau(x[\tau]) = h_\tau(x[\tau],u[\tau]) + V_{\tau + 1}(f(x[\tau],u[\tau])).
\label{eqn:DP-cost-to-go}
\end{align}



\subsection{Useful Lemmas}
Below are some useful lemmas that we will use in our derivations later. We omit the proofs as they are all trivial.

\begin{lemma}[Lemma 1 in \cite{tseng2020system}]\label{lem:null-space-is-subspace}
Given a matrix $\Psi$, $\nul{\Psi}$ is a subspace and there exists some matrix $\Xi = \basis{\nul{\Psi}}$.
\end{lemma}

\begin{lemma}[Lemma 2 in \cite{tseng2020system}]\label{lem:union-of-null-space}
The intersection of $\nul{\Psi_a}$ and $\nul{\Psi_b}$ is $\nul{\mat{\Psi_a\\ \Psi_b}}$.
\end{lemma}

\begin{lemma}\label{lem:pinv-equal-regression}
Given matrices $\Gamma$ and $Z$, we have
\begin{align*}
\Gamma^+ Z = \argmin\limits_{M} \norm{\Gamma M - Z}.
\end{align*}
\end{lemma}

\begin{lemma}\label{lem:normalization}
Letting $\Psi$ be a matrix with at least one non-zero row, we have
$\nul{\Psi} = \nul{\normalize{\Psi}}$.
\end{lemma}

\section{DP Algorithms for Output-Feedback SLS}\label{sec:DP_algorithm} 

In the following, we provide our DP algorithms to solve the output-feedback SLS \eqn{SLS_output_feedback} with the set of system level constraints ($\Scal$).
We first reformulate the SLS problem as a control problem, which then allows us to apply DP to derive explicit and approximation solutions efficiently.

\subsection{Control Problem Reformulation}
We can reformulate the SLS optimization \eqn{SLS_output_feedback} as a control problem by treating the system response $\OFqple$ as the state $\xbf$ and the input $\ubf$ of a linear system.
Accordingly, we can rewrite the cost function $g(\Pbfxx,\Pbfxy,\Pbfux,\Pbfuy)$ in terms of the state and input as $h(\xbf,\ubf)$, resulting in a state-feedback control problem. As such, we can apply DP to solve the control problem, which equivalently solves the original SLS optimization. Below, we describe the reformulation in detail.

Let $\Vec{A}$ be the stacked column vectors of a matrix $A$, we slightly abuse the notation to overwrite/define
\begin{align*}
x[\tau] = \mat{
x_{xx}[\tau] \\
x_{xy}[\tau] \\
x_{ux}[\tau]
}
= \mat{
\Vec{\Pxx[\tau]}\\ 
\Vec{\Pxy[\tau]}\\
\Vec{\Pux[\tau]}
}, \quad
u[\tau] = \Vec{\Puy[\tau]},
\end{align*}
and define the per-step cost $h_\tau$ to be
\begin{align*}
h_\tau(x[\tau],u[\tau]) = g_{\tau}(\Pxx[\tau],\Pxy[\tau],\Pux[\tau],\Puy[\tau]).
\end{align*}
Since we consider only FIR system response with horizon $T$, the constraint set in \eqn{SLS_output_feedback} can be written as
\begin{align*}
\Pxx[\tau + 1] =&\ \Pxx[\tau] A + \Pxy[\tau] C\\
\Pxy[\tau + 1] =&\ A \Pxy[\tau] + B \Puy[\tau]\\
\Pux[\tau + 1] =&\ \Pux[\tau] A + \Puy[\tau] C
\end{align*}
for all $\tau = 1,\dots,T - 1$, 
\begin{align*}
A \Pxx[\tau] + B \Pux[\tau] = \Pxx[\tau] A + \Pxy[\tau] C
\end{align*}
for all $\tau = 1,\dots,T$, and 
\begin{gather*}
\Pxx[0] = 0, \quad \Pxy[0] = 0, \quad \Pux[0] = 0,\\
\Pxx[1] = I, \quad \Pxy[1] = B \Puy[0], \quad \Pux[1] = \Puy[0] C
\end{gather*}
We can rewrite the above equations in terms of $x[\tau]$ and $u[\tau]$. Let $\Scal_x$ and $\Scal_u$ be sets such that if $\Phibf \in \Scal$ then $\xbf \in \Scal_x$ and $\ubf \in \Scal_u$. This reformulates \eqn{SLS_output_feedback} as 
\OptMin[reformulated_opt]{ 
\sum\limits_{\tau = 0}^T h_\tau(x[\tau],u[\tau])
}{
\OptCons{
x[\tau+1] = \tdA x[\tau] + \tdB u[\tau]
}{\forall \tau = 1,..., T-1 }{of-state-dynamics}\\
\OptCons{
\tdA_{eq} x[\tau] = 0
}{\forall \tau = 1,..., T }{of-transition-constraint}\\
\OptCons{
x[0] = 0, x[1] = \mat{\vec{I}\\ \tdB_0 u[0] }
}{}{of-initial-condition}\\
\OptCons{
\tilde{A} x[T] + \tilde{B} u[T] = 0
}{}{of-boundary-condition} \\
\OptCons{
x[\tau] \in \Scal_x[\tau],\  u[\tau] \in \Scal_u[\tau]
}{}{of-system-level-constraints}
}
where $\tdA$, $\tdB$, and $\tdA_{eq}$ are derived according to the SLS conditions. 

Notice that given $x[\tau]$ as the state and $u[\tau]$ as the control at time $\tau$, \eqn{reformulated_opt} can be seen as a discrete-time state-feedback control problem with state dynamics \eqn{of-state-dynamics}, transition constraint \eqn{of-transition-constraint}, initial condition \eqn{of-initial-condition}, boundary condition \eqn{of-boundary-condition}, and system-level constraints \eqn{of-system-level-constraints}.

\subsection{Explicit Solution by Dynamic Programming}
Since \eqn{reformulated_opt} is a state-feedback problem, we can adopt a similar procedure as in \cite{tseng2020system} to solve it. Specifically, we need to 
\begin{itemize}
\item {\it backward recursion:} recursively compute $V_\tau(x[\tau])$ backwards in $\tau$ based on \eqn{of-state-dynamics};
\item {\it transition constraint and boundary condition:} enforce \eqn{of-transition-constraint}, \eqn{of-boundary-condition}, and \eqn{of-system-level-constraints} throughout the derivation.
\item {\it initial condition:} enforce \eqn{of-initial-condition}.
\end{itemize}
In this paper, we show how to derive the DP procedure for unconstrained SLS, i.e., \eqn{reformulated_opt} without system-level constraints \eqn{of-system-level-constraints}, and we could easily extend our results here to handle entrywise linear \eqn{of-system-level-constraints} using similar techniques as in \cite[Section II C]{tseng2020system}.

We examine these three aspects below. 

\paragraph{Backward recursion}
For the backward recursion, we can set
\begin{align*}
f(x[\tau], u[\tau]) = \tdA x[\tau] + \tdB u[\tau]
\end{align*}
and apply the procedure in \sec{preliminaries-DP} to derive $V_\tau(x[\tau])$ and $u[\tau]$. 

\paragraph{Transition constraint, system constraint, and boundary condition}
We enforce \eqn{of-transition-constraint} and \eqn{of-boundary-condition} during the recursion, which requires maintaining the admissible input set $\Acal_u[\tau]$ as shown in \cite{tseng2020system}. A result in \cite{tseng2020system} iteratively shows that $x[\tau]$ lies in the null space of some matrix $\Psi_x[\tau]$ and derives $\Acal_u[\tau]$ accordingly. 
However, we cannot directly adopt the results from \cite{tseng2020system} to derive $\Acal_u[\tau]$ as we have the new condition \eqn{of-transition-constraint} to meet. 
Fortunately, \cite{tseng2020system} combines its Corollary 1 and equation (11) to deal with additional conditions for $x[\tau]$ (the entrywise linear condition), and we can borrow such a concept to derive the following theorem for our output-feedback scenario to handle the new condition \eqn{of-transition-constraint}.
\begin{theorem}\label{thm:feasible_set_unstable}
  Suppose $x[\tau+1] \in \nul{\Psi_x[\tau+1]}$ for some given matrix $\Psi_x[\tau+1]$. For $\tau \geq 1$, we have 
  \begin{align*}
    \Acal_u[\tau] = \{ \hat{u} : \hat{u} = H_x x[\tau] + H_\lambda \lambda \}
  \end{align*}
  where 
  \begin{align*}
    \Gamma =&\ \mat{-\tdB & \Xi_x[\tau + 1]}, \\
    H_x =&\ \mat{I & 0} \Gamma^+ \tdA, \\
    H_\lambda =&\ \mat{I & 0} (I - \Gamma^+ \Gamma),
  \end{align*}
  and $\Xi_x[\tau+1] = \basis{\nul{\Psi_x[\tau+1]}}$.

  Also, $x[\tau] \in \nul{\Psi_x[\tau]}$ where 
  \begin{align*}
    \Psi_x[\tau] = \mat{
      (\Gamma \Gamma^+ - I)\tdA\\
      \tdA_{eq}
    }.
  \end{align*}
\end{theorem}

\begin{proof}
Without condition \eqn{of-transition-constraint}, the former part follows directly from \cite[Corollary 1]{tseng2020system}. Since \eqn{of-transition-constraint} is equivalent to $x[\tau] \in \nul{\tdA_{eq}}$, the theorem follows from \lem{union-of-null-space}.
\end{proof}

Using \thm{feasible_set_unstable}, we can define $\Psi_x[T + 1] = I$ and backward-recursively derive each $\Acal_u[\tau]$ and $\Psi_x[\tau]$ to enforce \eqn{of-transition-constraint} and \eqn{of-boundary-condition}. Though \thm{feasible_set_unstable} works in theory, in practice the computation of the pseudo-inverse $\Gamma^+$ can hardly be done precisely. The numerical error then affects the precision of the matrix $\Psi_x$, which further disturbs the resulting basis $\Xi_x$ and causes numerical instability. Therefore, we propose the following theorem that leads to a more numerically stable derivation of $H_x$ and $H_\lambda$.

\begin{theorem}\label{thm:feasible_set}
  Let 
\begin{align*}
\Gamma_A = \Psi_x[\tau+1] \tdA, \quad \Gamma_B = \Psi_x[\tau+1] \tdB.
\end{align*}
  $H_x$ and $H_\lambda$ in \thm{feasible_set_unstable} can be alternatively computed by 
  \begin{align*}
    H_x =&\ \argmin_{M} \norm{\Gamma_B M + \Gamma_A},\\
    H_\lambda =&\ \basis{\nul{\Gamma_B}}.
  \end{align*}
 Also, $x[\tau] \in \nul{\Psi_x[\tau]}$ where 
  \begin{align*}
    \Psi_x[\tau] = \normalize{\mat{
      \Gamma_B H_x + \Gamma_A\\
      \tdA_{eq}
    }}.
  \end{align*}
\end{theorem}

\begin{proof}
Since $x[\tau+1] \in \nul{\Psi_x[\tau+1]}$, we have
\begin{align*}
\Psi_x[\tau+1] x[\tau+1] =&\
\Psi_x[\tau+1] \left(\tdA x[\tau] + \tdB u[\tau] \right) \\
=&\ \Gamma_A x[\tau] + \Gamma_B u[\tau] = 0.
\end{align*}
Therefore, we know 
\begin{align*}
 \Gamma_B u[\tau&] = - \Gamma_A x[\tau]\\
\Rightarrow u[\tau] =&\ - \Gamma_B^+ \Gamma_A x[\tau] + (I - \Gamma_B^+ \Gamma_B) \lambda'\\
=&\ H_x x[\tau] + H_{\lambda'} \lambda'.
\end{align*}
$H_x$ can then be computed by \lem{pinv-equal-regression}. On the other hand, we know $\Gamma_B H_{\lambda'} \lambda' = 0$. Therefore, $H_{\lambda'} \lambda' \in \nul{\Gamma_B}$, and we can reparameterize the space by setting $H_\lambda = \basis{\nul{\Gamma_B}}$. Lastly, the solution to the above equation exists if and only if
\begin{align*}
(- \Gamma_B \Gamma_B^+ \Gamma_A + \Gamma_A) x[\tau] = (\Gamma_B H_x + \Gamma_A) x[\tau] = 0.
\end{align*}
Therefore, $x[\tau] \in \nul{\Gamma_B H_x + \Gamma_A}$. Along with $x[\tau] \in \nul{\tdA_{eq}}$, we can derive $\Psi_x[\tau]$ via \lem{union-of-null-space} and \lem{normalization}.
\end{proof}

\thm{feasible_set} provides a numerically more stable procedure to compute $H_x$ and $H_\lambda$, and while we leave this proof for future work, our experiments indicate that
\begin{itemize}
\item $H_x$ can be computed by regression, which can be done via some sophisticated algorithms;
\item $H_\lambda$ have a smaller dimension;
\item the normalized $\Psi_x$ is less likely to suffer resolution issues. 
\end{itemize}

\paragraph{Initial condition} Under the state-feedback scenario in \cite{tseng2020system}, we only need to set $\Pxx[1] = I$ at the end to enforce the initial condition. However, under our output-feedback setting, we have two sets of initial conditions. We need to ensure both $x[0]$ and $x[1]$ in \eqn{of-initial-condition}. Therefore, we partition the whole backward recursion into two phases: from $T$ to $1$ and from $1$ to $0$. For the former phase, we derive $\Acal_u[\tau]$ and $\Psi_x[\tau]$ recursively according to \thm{feasible_set} until we obtain $\Psi_x[1]$. In the next phase, we derive $\Acal_u[0]$ using the following theorem.

\begin{theorem}\label{thm:feasible_set0}
  Suppose $x[1] \in \nul{\Psi_x[1]}$ for some given matrix $\Psi_x[1]$ where
  $
  \Psi_x[1] = \mat{ \Psi_{x_1}[1] & \Psi_{x_2}[1] }
  $
  such that 
  \begin{align*}
  \Psi_{x_1}[1] x_{xx}[1] + \Psi_{x_2}[1] \mat{x_{xy}[1] \\ x_{ux}[1]} = 0.
  \end{align*}
   We have
  \begin{align*}
    \Acal_u[0] = \{ \hat{u} : \hat{u} = w + H_\lambda \lambda \}
  \end{align*}
  where 
  \begin{align*}
    \Gamma =&\ \Psi_{x_2}[1] \tdB_0, \\ 
    w =&\ \argmin_{M} \norm{\Gamma + \Psi_{x_1}[1] \vec{I}}, \\
    H_\lambda =&\ \basis{\nul{\Gamma}}.
  \end{align*}
\end{theorem}

\begin{proof}
Since $x[1] \in \nul{\Psi_x[1]}$, we have 
\begin{align*}
\Psi_x[1] x[1] = 0 
=&\ \Psi_x[1] \mat{\vec{I}\\ \tdB_0 u[0] }\\
=&\ \Psi_{x_1}[1] \vec{I} + \Psi_{x_2}[1] \tdB_0 u[0].
\end{align*}
Rearrange the terms and we obtain
\begin{align*}
 \Psi_{x_2}[1] \tdB_0 u[0] = \Gamma u[0] = - \Psi_{x_1}[1] \vec{I}.
\end{align*}
Therefore,
\begin{align*}
u[0] = - \Gamma^+ \Psi_{x_1}[1] \vec{I} + (I - \Gamma^+\Gamma) \lambda.
\end{align*}
We can then show the results following the same procedure in the proof of \thm{feasible_set}.
\end{proof}

Combining both \thm{feasible_set} and \thm{feasible_set0}, we can derive admissible input sets $\Acal_u[\tau]$ for all $\tau = 0, \dots, T$. Let $\unvec{x}$ rebuild a matrix $A$ such that $\vec{A} = x$, we present our DP output-feedback SLS in \alg{plainSLS}.

\begin{algorithm}
\caption{DP for output feedback SLS}\label{alg:plainSLS}
\begin{algorithmic}[1]
    \REQUIRE Per-step costs $g_{\tau}(\Pxx[\tau],\Pxy[\tau],\Pux[\tau],\Puy[\tau])$ for all $\tau = 0, \dots, T$ and the matrices $A$, $B$, and $C$ in \eqn{SLS_output_feedback}.
    \ENSURE $\Pxx[\tau],\Pxy[\tau],\Pux[\tau],\Puy[\tau]$ for all $\tau = 0, \dots, T$.
    \STATE Derive $h_\tau$ from $g_{\tau}$ for all $\tau = 0, \dots, T$.\label{plainSLS-setup0}
    \STATE Derive $\tdA$, $\tdB$, $\tdA_{eq}$, and $\tdB_0$ from $A$, $B$ and $C$ according to \eqn{SLS_output_feedback}.
    \STATE $\Psi_x[T+1] = I$.
    \STATE $V_{T+1}(x[T+1]) = 0$.\label{plainSLS-setup1}
    \FOR{$\tau=T,\dots,1$}
        \STATE Derive $\Acal_u[\tau]$ and $\Psi_x[\tau]$ by \thm{feasible_set}.\label{plainSLS-compute0}
        \STATE Compute $K_{\tau}(x[\tau])$ by \eqn{DP-U}.
        \STATE Derive $V_{\tau}(x_{\tau})$ by \eqn{DP-cost-to-go}.\label{plainSLS-compute1}
    \ENDFOR
    \STATE Derive $\Acal_u[0]$ by \thm{feasible_set0}.\label{plainSLS-final0}
    \STATE Compute $u[0]$ by \eqn{DP-U}.
	\STATE Set $\Puy[0] = \unvec{u[0]}$ and $\Pxx[0] = \Pxy[0] = \Pux[0] = 0$.
    \STATE Compute $x[1]$ by \eqn{of-initial-condition}.    
    \FOR{$\tau=1,\dots,T$}\label{plainSLS-final0.0}
        \STATE $u[\tau] = K_{\tau}(x[\tau]).$
        \STATE Set $\Pxx[\tau] = \unvec{x_{xx}[\tau]}$, $\Pxy[\tau] = \unvec{x_{xy}[\tau]}$, $\Pux[\tau] = \unvec{x_{ux}[\tau]}$, and $\Puy[\tau] = \unvec{u[\tau]}$.
		\STATE $x[\tau+1] = \tdA x[\tau] + \tdB u[\tau]$.        
    \ENDFOR\label{plainSLS-final1}
\end{algorithmic}
\end{algorithm}





\subsection{Approximation for Faster Computation}
Though \thm{feasible_set} has greatly improved the numerical stability and simplified computation over \thm{feasible_set_unstable}, calculating the matrices in \thm{feasible_set} is still computationally expensive. Therefore, we explore the possibility of an approximation algorithm which trades the precision for faster computation. Inspired by the approximation to the infinite horizon SLS problem
in \cite{tseng2020system}, we can relax \eqn{of-transition-constraint} and \eqn{of-boundary-condition} by using
\begin{align*}
H_x = 0, \quad H_\lambda = I,
\end{align*}
which is equivalent to an unconstrained $\Acal_u$, and reuse $\Psi_x[T_a]$ as $\Psi_x[1]$ for 
\thm{feasible_set0} where $T_a$ is the \emph{allowance}, that is, the number of time steps at which we use
an unconstrained $\Acal_u$. We summarize this approximation in \alg{approxDP}. We observe imperically that the null space defining $\Acal_u$ is close enough to $\Rmbb^n$ that it is practically useful; we defer proof of this to future work.

\begin{algorithm}
\caption{DP Approximation for output feedback SLS}\label{alg:approxDP}
\begin{algorithmic}[1]
    \REQUIRE Per-step costs $g_{\tau}(\Pxx[\tau],\Pxy[\tau],\Pux[\tau],\Puy[\tau])$ for all $\tau = 0, \dots, T$, the matrices $A$, $B$, and $C$ in \eqn{SLS_output_feedback}, and allowance $T_a$.
    \ENSURE $\Pxx[\tau],\Pxy[\tau],\Pux[\tau],\Puy[\tau]$ for all $\tau = 0, \dots, T$.
    \STATE Follow line \ref{plainSLS-setup0} to \ref{plainSLS-setup1} in \alg{plainSLS}.
    \FOR{$\tau=T,\dots,T_a$}
        \STATE Perform line \ref{plainSLS-compute0} to \ref{plainSLS-compute1} in \alg{plainSLS}.
    \ENDFOR
	
    \FOR{$\tau=T_a,\dots,1$}
        \STATE Compute $K_{\tau}(x[\tau])$ by \eqn{DP-U} with unconstrained $\Acal_u[\tau]$.
        \STATE Derive $V_{\tau}(x_{\tau})$ by \eqn{DP-cost-to-go}.
    \ENDFOR
    \STATE Let $\Psi_x[1] = \Psi_x[T_a]$.
    \STATE Perform line \ref{plainSLS-final0} to \ref{plainSLS-final1} in \alg{plainSLS}.
\end{algorithmic}
\end{algorithm}

\section{Example Algorithms}\label{sec:examples}
Now that we have defined the dynamic programming algorithms, we illustrate results for two classes of objective functions.
We show algorithms for computing the optimal control under two cost functions: a $\mathcal{H}_2$ and a quadratic cost.
Define the $\mathcal{H}_2$ cost to be 
\begin{align*} 
    h_{\tau}(x[\tau],u[\tau]) = \norm{ Fx[\tau]+Gu[\tau] }_F^2,
\end{align*}
where $\norm{\cdot}_F$ is the Frobenius norm, and the quadratic cost to be 
\begin{align*} 
    h_{\tau}(x[\tau],u[\tau]) = x[\tau]^\top  Q x[\tau] + u[\tau]^\top  R u[\tau],
\end{align*}
where $Q$ is a positive definite matrix weighing the cost of elements of the state. Note that the $\Hcal_2$ objective is a generalization of the 
quadratic objective where cross terms between $x$ and $u$ are permitted.

\subsection{$\mathcal{H}_2$ Objective}
Using the reformulated SLS contraint and optimization equations \eqn{reformulated_opt},
we implement a dynamic programming algorithm by deriving an explicit 
expression for the cost-to-go function from the objective function $g_{\tau}$. For $\tau > 0$, we have:
\begin{align*} 
V_{\tau}(x[\tau]) = \min\limits_{\hat{u} \in \Acal_u[\tau]} \norm{ Fx[\tau]+ G \hat{u} }_F^2 +
V_{\tau+1}(f(x[\tau],\hat{u})).
\end{align*}
Using the form of $U$ from Theorem 1 and the claim in equation 12 in \cite{tseng2020system}, 
we know that the cost-to-go function can be parameterized by some matrix $P$.
\begin{equation*}
    \begin{aligned}
        V_{\tau}(x[\tau]) = \min_{\lambda} \{ \norm{ F_x x[\tau] + G_{\lambda} \lambda }_F^2 + \\
        \sum_{t=\tau+1}^\top  \norm{ P[\tau] (A_x x[\tau] + B_{\lambda} \lambda ) }_F^2 \} \\
    \end{aligned}
\end{equation*}
where
\begin{align*}
    A_x =&\ A + BH_x, & B_{\lambda}=&\ BH_{\lambda},\\
    F_x =&\ F + GH_x, & G_{\lambda}=&\ GH_{\lambda}.
\end{align*}

After solving for the $\lambda$ that minimizes the function, which then gives a 
controller as a function of the state, the cost-to-go is given by
\begin{equation*}
    \begin{aligned}
        V_{\tau}(x[\tau]) = \min_{\lambda} \{ \norm{ (F + G K[\tau])x[\tau] }_F^2 + \\
        \sum_{t=\tau+1}^\top  \norm{ P[\tau] (A + BK[\tau])x[\tau] }_F^2 \}.
    \end{aligned}
\end{equation*}
where 
\begin{gather}
        K[\tau] = H_x + H_{\lambda} L[\tau] \label{eqn:Kt_H2} \\
        L[\tau] = -(G_{\lambda}^\top G_{\lambda}+B_{\lambda}^\top P[\tau]B_{\lambda})^{-1}(G_{\lambda}^\top  F_x +B_{\lambda}^\top  P[\tau]A_x). \nonumber
\end{gather}
\cite{tseng2020system} shows how to compute $P[\tau]$; we state the result here:
\begin{align}\label{eqn:update_p}
    \begin{gathered}
        P[\tau-1] = (C+GK[\tau])^\top (C+GK[\tau]) \\
    + (A+BK[\tau])^\top P[\tau](A+BK[\tau]).
    \end{gathered}
\end{align}
The derivation changes from the state feedback scenario for $\tau=0$. The SLS constraints 
that hold for the previous transitions change, so the cost-to-go is instead given as an optimization over
a control input that influences $\Pxy$, $\Pux$, and $\Puy$. Recall $P[1]$ as the matrix defining the cost-to-go
at the next state, where the cost-to-go at state $x[1]$ is $V_1(x[1]) = \norm{ P[1] x[1] }_F^2$.
\begin{equation*} 
    \begin{aligned}
        V_0(x[0]) = \min_{\hat{u} \in \Acal_u[0]} \{ \norm{ G \hat{u} }_F^2 + \norm{ P[1]\mat{\Vec{I} \\ \tilde{B}_0\hat{u}}  }_F^2 \}
    \end{aligned}
\end{equation*}
Now we must consider the form that $u[0]$ will take. We showed in \thm{feasible_set0} that 
$\Acal_u[0]$ consists of vectors of the form $\hat{u} = H_{\lambda}\lambda + w$. We can write the cost-to-go as
\begin{equation*}
    \begin{aligned}
        V_0(x[0]) = \min_{\lambda} \{ \norm{ G (H_{\lambda}\lambda + w) }_F^2 + \\
        \norm{ P[1]\mat{\Vec{I}\\ \tilde{B}_0(H_{\lambda}\lambda + w)} }_F^2 \}
    \end{aligned}
\end{equation*}
and the control parameterization as
\begin{equation}\label{eqn:K0_H2}
    \begin{gathered}
        u[0] = H_{\lambda}\lambda^* + w,\\
        \lambda^* = \argmin_{\lambda} \{ \norm{ G (H_{\lambda}\lambda + w) }_F^2 + \\
        \norm{ P[1] \mat{\Vec{I}\\ \tilde{B}_0(H_{\lambda}\lambda + w)} }_F^2 \}.
    \end{gathered}
\end{equation}
The solution can be found by setting the derivative equal to zero and solving for $\lambda^*$. The steps are shown 
in \alg{H2}.

\begin{algorithm}
    \caption{DP with $\mathcal{H}_2$ Objective}
    \label{alg:H2}
    \begin{algorithmic}[1]
    \REQUIRE Per-step costs $g_{\tau}(\Pxx[\tau],\Pxy[\tau],\Pux[\tau],\Puy[\tau])$ for all $\tau = 0, \dots, T$ and the matrices $A$, $B$, and $C$ in \eqn{SLS_output_feedback}; $F$ and $G$.
    \ENSURE $\Pxx[\tau],\Pxy[\tau],\Pux[\tau],\Puy[\tau]$ for all $\tau = 0, \dots, T$.
    \STATE Derive $h_\tau$ from $g_{\tau}$ for all $\tau = 0, \dots, T$.
    \STATE Derive $\tdA$, $\tdB$, $\tdA_{eq}$, and $\tdB_0$ from $A$, $B$ and $C$ according to \eqn{SLS_output_feedback}.
    \STATE $\Psi_x[T+1] = I$.
    \STATE $P[T]=0$. 
    \FOR{$\tau=T,\dots,1$}
    \STATE Derive $\mathcal{A}_u$ and $\nullspace(\Psi_x[\tau])$ by \thm{feasible_set}.
    \STATE Compute $K_{\tau}(x_{\tau})$ by \eqn{Kt_H2}.
    \STATE Compute $P[\tau-1]$ from $P[\tau]$ by \eqn{update_p}.
    \ENDFOR
    \STATE Derive $\Acal_u[0]$ by \thm{feasible_set0}.
    \STATE Compute $u[0]$ by \eqn{K0_H2}.
	\STATE Set $\Puy[0] = \unvec{u[0]}$ and $\Pxx[0] = \Pxy[0] = \Pux[0] = 0$.
    \STATE Compute $x[1]$ by \eqn{of-initial-condition}.
    \STATE Perform line \ref{plainSLS-final0.0} to \ref{plainSLS-final1} in \alg{plainSLS}.
    \end{algorithmic}
\end{algorithm}

\subsection{Quadratic Objective}
Now we consider the algorithm under the quadratic objective. We derive an explicit cost-to-go function under 
this objective as
\begin{align*}
V_{\tau}(x[\tau]) = \min\limits_{\hat{u} \in \Acal_u[\tau]} x[\tau]^\top  Q x[\tau] + \hat{u}^\top  R \hat{u} + V_{\tau+1}(f(x[\tau],\hat{u})).
\end{align*}
Starting from $\tau = T$ and recursively computing backward, from \cite{tseng2020system} we have that $V_{\tau}(x[\tau])$ can be written as $x[\tau]^\top  V_{(\tau)} x[\tau]$.
We know the future cost at time $\tau=T$ is zero, and $u[T]$ takes the form $u[T]=H_{\lambda}\lambda + H_x x[T]$, so we can write the overall cost-to-go as 
\begin{align*}
    \begin{gathered}
        V_{T}(x[T])
        = \min\limits_{\lambda} x[T]^\top  Q x[T] \\
        + (H_{\lambda}\lambda + H_x x[T])^\top  R (H_{\lambda}\lambda + H_x x[T]).
    \end{gathered}
\end{align*}
After solving for the value of $\lambda$ that minimizes $V_T(x[T])$, and assuming that $R$ is symmetric, we can write the cost as
\begin{align*}
    \begin{gathered}
        V_T(x[T]) = x[T]^\top  V_{(T)} x[T] \\
        \text{where }\ V_{(T)}= Q + \bar{Q} \\
        \bar{Q} = \tilde{Q}^\top  R \tilde{Q} \\
        \tilde{Q} = I-H_{\lambda}(H_{\lambda}^\top RH_{\lambda})^{-1}H_x.
    \end{gathered}
\end{align*}
Each $V_{\tau}(x[\tau])$ can also be written in this quadratic form. Pugging in the form of $u[\tau]$ and the 
quadratic form of the future cost-to-go function, we set the derivative equal to zero and solve for 
the $\lambda$ which minimizes $V_\tau(x[\tau])$, resulting in the following definitions for 
$V_{\tau}(x[\tau])$ and $K[\tau]$:
\begin{align}\label{eqn:V_LQ}
    \begin{gathered}
        V(x[\tau]) = x[\tau]^\top  V_{(\tau)} x[\tau]\ \ \text{where }\\
        V_{(\tau)} = Q + K[\tau]^\top  R K[\tau] + \\
        (\tilde{A}+\tilde{B}K[\tau])^\top  V_{(\tau+1)}(\tilde{A}+\tilde{B}K[\tau])
    \end{gathered}
\end{align}
and
\begin{align}\label{eqn:K_LQ}
    \begin{gathered}
        K[\tau] = H_x - H_{\lambda} L_d^{-1} L_n\ \ \text{where } \\
        L_d = H_{\lambda}^\top  R H_{\lambda} + (\tilde{B} H_{\lambda})^\top  V_{(t+1)} \tilde{B} H_{\lambda} \\
        L_n = H_{\lambda}^\top  R H_x + (\tilde{B} H_{\lambda})^\top  V_{(t+1)} (\tilde{A}+\tilde{B}H_x).
    \end{gathered}
\end{align}
Lastly, we compute the control input for the special case where $\tau = 0$. The input $u[0]$ takes the form 
from \thm{feasible_set0}, and we know the state at time $t=0$ is zero. We have the control  given by 
\begin{align}\label{eqn:K0_LQ}
    \begin{gathered}
        u[0] = H_{\lambda}\lambda^* + w \\
        \lambda^* = \argmin_{\lambda} (H_{\lambda}\lambda + w)^\top  R (H_{\lambda}\lambda + w) \\
        + \mat{\vec{I}\\ \tilde{B}_0(H_{\lambda}\lambda + w)}^\top  V_{(1)} \mat{\vec{I}\\ \tilde{B}_0(H_{\lambda}\lambda + w)}
    \end{gathered}
\end{align}
The steps are enumerated in \alg{LQ}.
\begin{algorithm}
    \caption{DP with Quadratic Objective}
    \label{alg:LQ}
    \begin{algorithmic}[1]
    \REQUIRE Per-step costs $g_{\tau}(\Pxx[\tau],\Pxy[\tau],\Pux[\tau],\Puy[\tau])$ for all $\tau = 0, \dots, T$ and the matrices $A$, $B$, $C$ in \eqn{SLS_output_feedback}; $Q$ and $R$.
    \ENSURE $\Pxx[\tau],\Pxy[\tau],\Pux[\tau],\Puy[\tau]$ for all $\tau = 0, \dots, T$.
    \STATE Derive $h_\tau$ from $g_{\tau}$ for all $\tau = 0, \dots, T$.
    \STATE Derive $\tdA$, $\tdB$, $\tdA_{eq}$, and $\tdB_0$ from $A$, $B$ and $C$ according to \eqn{SLS_output_feedback}.
    \STATE $\Psi_x[T+1] = I$.
    \STATE $V_{T+1}(x[T+1]) = 0$.
    \FOR{$\tau=T,\dots,1$}
        \STATE Derive $\Acal_u[\tau]$ and $\Psi_x[\tau]$ by \thm{feasible_set}.
        \STATE Compute $K_{\tau}(x[\tau])$ by \eqn{K_LQ}.
        \STATE Derive $V_{\tau}(x_{\tau})$ by \eqn{V_LQ}.
    \ENDFOR
    \STATE Derive $\Acal_u[0]$ by \thm{feasible_set0}.
    \STATE Compute $u[0]$ by \eqn{K0_LQ}.
	\STATE Set $\Puy[0] = \unvec{u[0]}$ and $\Pxx[0] = \Pxy[0] = \Pux[0] = 0$.
    \STATE Compute $x[1]$ by \eqn{of-initial-condition}.
    \STATE Perform line \ref{plainSLS-final0.0} to \ref{plainSLS-final1} in \alg{plainSLS}.
    \end{algorithmic}
\end{algorithm}

\section{Evaluation}\label{sec:evaluation}
To evaluate the performance of our algorithm, we compare the time required to run our nominal algorithm and
approximation algorithm with the solution given by CVX. We construct a set of stochastic chain matrices and time
 the computations for the $\mathcal{H}_2$ and quadratic objectives for the CVX algorithm, dynamic programming (\alg{H2} and \alg{LQ}), and their corresponding approximation algorithm (in the form of \alg{approxDP}). 
Additionally, we determine in which scenarios the approximation algorithm produces transfer functions that successfully attenuate disturbances. 
The data was generated on a desktop with an AMD Ryzen 7 3700X processor (16 logical cores) and 32 GB DDR4 memory; the CVX code was taken from SLSpy \cite{SLSpy}.

\subsection{Evaluation Setup}
The test system is a stochastic chain with state dimension $N_x$, input dimension $N_u$, and observation dimension $N_y$, which takes the form
\begin{align*}
    A = \begin{bmatrix}
        1-\alpha & \alpha & & & \\
        \alpha & 1-2\alpha & \alpha & & \\
        & \ddots & \ddots & \ddots & \\
        & & \alpha & 1-2\alpha & \alpha \\
        & & & \alpha & 1-\alpha
    \end{bmatrix}\\
    B = \begin{bmatrix}
        I_{N_u} \\
        0
    \end{bmatrix}\ \ 
    C = \begin{bmatrix}
        I_{N_y} & 0
    \end{bmatrix}
\end{align*}
where $\alpha \in (0,1)$ and $I_{N\cdot}$ is the identity matix of size $N\cdot$. We set $D$ to zero and do not use a constraint set. 
We construct the test cases to span the size ranges $N_x \in [5,20]$ with $N_u = N_y = N_x$ and $T=10$. 
Let $T$ be the length of the finite impulse response (FIR) horizon; we run scenarios with $T \in [10,25]$ with $N_x=N_u=N_y=10$. For both tests, we run 50 simulations per data point.
The computing time is defined as the time required to 
compute \alg{H2} or \alg{LQ}, which includes constructing the reformulated matrices $\tilde{A}$, $\tilde{B}$ and new cost function $h_{\tau}$. Analogously, the
computing time for CVX is the time required to synthesize the stablizing controller under \alg{H2} or \alg{LQ}.
\subsection{Computation Time and Scalability}
We plot the computing time as a function of each problem dimension ($N_x$ and $T$) for the $\mathcal{H}_2$ and quadratic objectives
to compare the algorithms' scalability. \fig{timing_all} illustrates how the DP algorithm reduces the computation time compared with
using CVX; for large values of $N_x$, DP is more than 5 times faster than CVX and for large $T$ it is at least 6 time faster. The greatest improvement is computed 
as over 7 times faster for $N_x$ for the quadratic objective. We also observe that the time increases more slowly with increasing system dimensions
compared with CVX for larger system dimensions, showing the scalability of our methods.

\begin{figure}
    \includegraphics[scale=0.5]{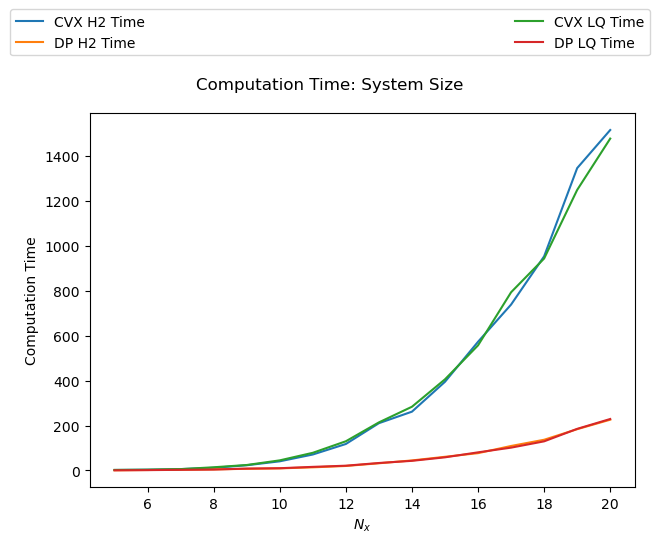}
    \includegraphics[scale=0.5]{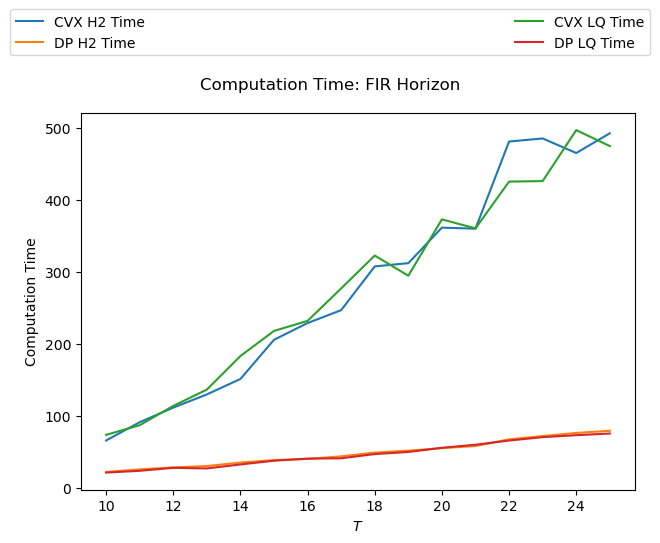}
    \caption{The dymanic programming algorithms outperforms the CVX algorithms for both the $\mathcal{H}_2$ and quadratic objectives.}
    \label{fig:timing_all}
\end{figure}

\subsection{Convergence of Approximation Algorithm}
The approximation algorithm partially removes the constraint requiring that the transfer functions at time $T+1$ are zero, which removes the guarantee 
that disturbances will be killed off in exchange for skipping expensive nullspace computations. We test scenarios with varying values of the allowance $T_a$, which is the
number of nullspace computations skipped, and record whether the transfer functions properly terminate disturbances within the finite impulse response time. 
In \fig{time_saved_skipped_iterations}, we see that avoiding computations of $\Acal_u$ for every time step in the FIR horizon significantly
improves the overall computation time relative to the plain DP algorithm; run time decrease by 42\% to 68\% for the largest allowances, and we see the largest percentage improvements for systems with the largest FIR horizon.
For all $T\in[10,25]$, for all $T - T_a \leq 3$, and for each of the 50 test simulations, we verified that the solution was in fact optimal.

\begin{figure}
    \includegraphics[scale=0.5]{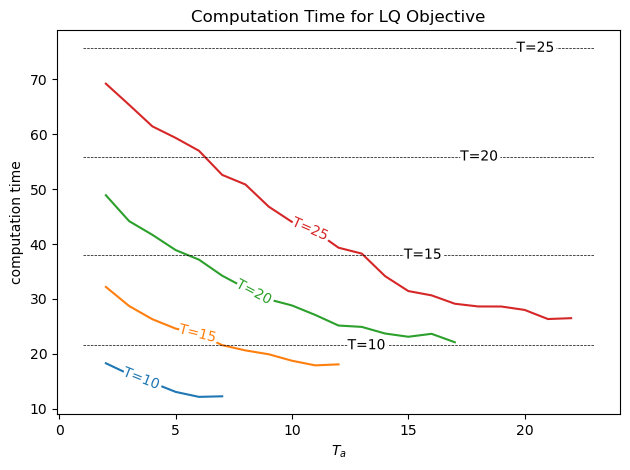}
    \includegraphics[scale=0.5]{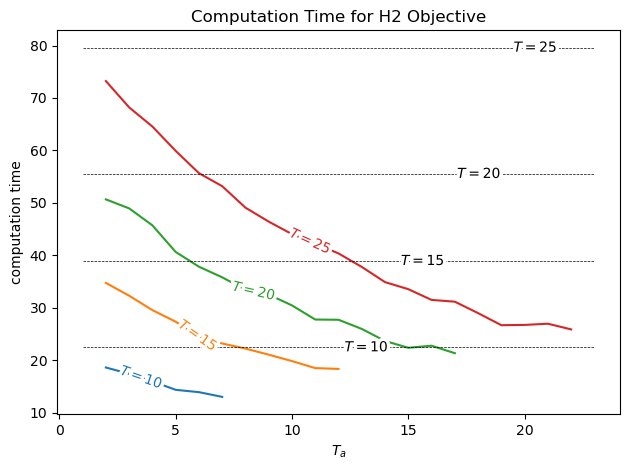}
    \caption{As the number of steps at which $H_X$ and $H_{\Lambda}$ are not recomputed increases, the time needed to compute the solution decreases.
    The dashed lines show the average computation time for the full SLS DP method.}
    \label{fig:time_saved_skipped_iterations}
\end{figure}

\section{Conclusion}\label{sec:conclusion}
We derived DP algorithms to solve output-feedback SLS problems by reformulating the optimization constraints as a linear control problem,
using two stages of DP to enforce boundary conditions and an interior condition. We illustrated specific examples of our DP algorithm on the $\mathcal{H}_2$ and quadratic objectives,
along with computation time results of each. The simulation results show that DP, including the approximation DP, outperforms the traditional CVX solver. 
Future work includes deriving either deterministic or probabilistic guarantees for when the approximation algorithm converges, given broader classes of systems beyond 
stochastic chains. Additionally, we want to investigate the robustness of this method to model uncertainty.

\bibliographystyle{IEEEtran}
\bibliography{%
bib_files/Robust_Control,%
bib_files/Robust_Optimization,%
bib_files/System_Level_Synthesis,%
bib_files/my_publications,%
bib_files/Textbook_and_Theory,%
bib_files/control_sequences,%
bib_files/IEEEabrv%
}

\begin{thebibliography}{10}
\providecommand{\url}[1]{#1}
\csname url@samestyle\endcsname
\providecommand{\newblock}{\relax}
\providecommand{\bibinfo}[2]{#2}
\providecommand{\BIBentrySTDinterwordspacing}{\spaceskip=0pt\relax}
\providecommand{\BIBentryALTinterwordstretchfactor}{4}
\providecommand{\BIBentryALTinterwordspacing}{\spaceskip=\fontdimen2\font plus
\BIBentryALTinterwordstretchfactor\fontdimen3\font minus
  \fontdimen4\font\relax}
\providecommand{\BIBforeignlanguage}[2]{{%
\expandafter\ifx\csname l@#1\endcsname\relax
\typeout{** WARNING: IEEEtran.bst: No hyphenation pattern has been}%
\typeout{** loaded for the language `#1'. Using the pattern for}%
\typeout{** the default language instead.}%
\else
\language=\csname l@#1\endcsname
\fi
#2}}
\providecommand{\BIBdecl}{\relax}
\BIBdecl

\bibitem{doyle1992feedback}
J.~C. Doyle, B.~A. Francis, and A.~R. Tannenbaum, \emph{Feedback Control
  Theory}.\hskip 1em plus 0.5em minus 0.4em\relax Courier Corporation, 1992.

\bibitem{zhou1996robust}
K.~Zhou, J.~C. Doyle, K.~Glover \emph{et~al.}, \emph{Robust and Optimal
  Control}.\hskip 1em plus 0.5em minus 0.4em\relax Prentice hall New Jersey,
  1996.

\bibitem{anderson2019system}
J.~Anderson, J.~C. Doyle, S.~H. Low, and N.~Matni, ``System level synthesis,''
  \emph{Annual Reviews in Control}, vol.~59, no.~12, pp. 3238--3251, 2019.

\bibitem{wang2019system}
Y.-S. Wang, N.~Matni, and J.~C. Doyle, ``A system level approach to controller
  synthesis,'' \emph{{IEEE} Trans. Autom. Control}, vol.~34, no.~8, pp.
  982--987, 2019.

\bibitem{furieri2019input}
L.~Furieri, Y.~Zheng, A.~Papachristodoulou, and M.~Kamgarpour, ``An
  input-output parametrization of stabilizing controllers: Amidst {Youla} and
  system level synthesis,'' \emph{IEEE Control Systems Letters}, vol.~3, no.~4,
  pp. 1014--1019, 2019.

\bibitem{sabau2014youla}
{\c{S}}.~Sab{\u{a}}u and N.~C. Martins, ``{Youla}-like parametrizations subject
  to {QI} subspace constraints,'' \emph{{IEEE} Trans. Autom. Control}, vol.~59,
  no.~6, pp. 1411--1422, 2014.

\bibitem{youla1976modern2}
D.~C. Youla, H.~A. Jabr, and J.~J. Bongiorno~Jr., ``Modern {Wiener}-{Hopf}
  design of optimal controllers -- part {II}: The multivariable case,''
  \emph{{IEEE} Trans. Autom. Control}, vol.~21, no.~3, pp. 319--338, 1976.

\bibitem{wang2018separable}
Y.-S. Wang, N.~Matni, and J.~C. Doyle, ``Separable and localized system-level
  synthesis for large-scale systems,'' \emph{{IEEE} Trans. Autom. Control},
  vol.~63, no.~12, pp. 4234--4249, 2018.

\bibitem{anderson2017structured}
J.~Anderson and N.~Matni, ``Structured state space realizations for {SLS}
  distributed controllers,'' in \emph{Proc. Allerton}, 2017, pp. 982--987.

\bibitem{matni2017scalable}
N.~Matni, Y.-S. Wang, and J.~Anderson, ``Scalable system level synthesis for
  virtually localizable systems,'' in \emph{Proc. {IEEE} {CDC}}.\hskip 1em plus
  0.5em minus 0.4em\relax IEEE, 2017, pp. 3473--3480.

\bibitem{amo-alonso2020distributed}
C.~Amo~Alonso and N.~Matni, ``Distributed and localized closed loop model
  predictive control via system level synthesis,'' in \emph{Proc. {IEEE}
  {CDC}}.\hskip 1em plus 0.5em minus 0.4em\relax IEEE, 2020, pp. 5598--5605.

\bibitem{amo-alonsosubeffective}
C.~Amo~Alonso and S.-H. Tseng, ``Effective {GPU} parallelization of distributed
  and localized model predictive control.''

\bibitem{tseng2020system}
S.-H. Tseng, C.~Amo~Alonso, and S.~Han, ``System level synthesis via dynamic
  programming,'' in \emph{Proc. IEEE CDC}, Dec. 2020.

\bibitem{SLSpy}
\BIBentryALTinterwordspacing
{SLSpy}. [Online]. Available: \url{{https://github.com/shih-hao-tseng/SLSpy}}
\BIBentrySTDinterwordspacing

\bibitem{tseng2021realization}
S.-H. Tseng, ``Realization, internal stability, and controller synthesis,'' in
  \emph{Proc. IEEE ACC}, May 2021.

\end{thebibliography}

\end{document}